\documentclass{birkjour}

 \newtheorem{thm}{Theorem}[section]
 \newtheorem{cor}[thm]{Corollary}
 \newtheorem{lem}[thm]{Lemma}
 \newtheorem{prop}[thm]{Proposition}
 \theoremstyle{definition}
 \newtheorem{defn}[thm]{Definition}
 \theoremstyle{remark}
 \newtheorem{rem}[thm]{Remark}
 
 \numberwithin{equation}{section}

\newtheorem{convention}{Convention}

\newcommand{\C}{\mathbb{C}}

\newcommand{\QQ}{\mathbb{Q}}

\newcommand{\BB}{\mathcal B}

\newcommand{\wt}{\widetilde}
\newcommand{\ima}{\hbox{Im}}

\newcommand{\rom}{\romannumeral}

\begin{document}

%-------------------------------------------------------------------------
% editorial commands: to be inserted by the editorial office
%
%\firstpage{1} \volume{228} \Copyrightyear{2004} \DOI{003-0001}
%
%
%\seriesextra{Just an add-on}
%\seriesextraline{This is the Concrete Title of this Book\br H.E. R and S.T.C. W, Eds.}
%
% for journals:
%
%\firstpage{1}
%\issuenumber{1}
%\Volumeandyear{1 (2004)}
%\Copyrightyear{2004}
%\DOI{003-xxxx-y}
%\Signet
%\commby{inhouse}
%\submitted{March 14, 2003}
%\received{March 16, 2000}
%\revised{June 1, 2000}
%\accepted{July 22, 2000}
%
%
%
%---------------------------------------------------------------------------
%Insert here the title, affiliations and abstract:
%

\title[Yet another version of Mumford's theorem]
 {Yet another version of Mumford's theorem}

%----------Author 1
\author[Robert Laterveer]{Robert Laterveer}

\address{
CNRS, Institut de Recherche Math\'ematiques Avanc\'ees\\
Universit\'e de Strasbourg\\
7 rue Ren\'e Descartes\\
F--67084 Strasbourg Cedex\\
France}

\email{laterv@math.unistra.fr}

\thanks{}
%----------classification, keywords, date


\date{}
%----------additions
%\dedicatory{To my boss}
%%% ----------------------------------------------------------------------

\begin{abstract}
The aim of this note is to provide a variant statement of Mumford's theorem. This variant states that for a general variety, all Chow groups are ``as large as possible'', in the sense that they cannot be supported on a divisor.
\end{abstract}

%%% ----------------------------------------------------------------------
\maketitle
%%% ----------------------------------------------------------------------
%\tableofcontents
\section{Introduction}

Mumford's theorem \cite{M} asserts that for a general variety over $\C$, the Chow group of $0$--cycles is very large. One version of Mumford's theorem states that for a variety $X$ with geometric genus $p_g(X)$ non--zero, the Chow group $A^nX$ is not supported on any closed subvariety:

\begin{thm}{(Bloch--Srinivas \cite{BS})} Let $X$ be a smooth projective variety of dimension $n$, and suppose $A^nX_{\QQ}$ is supported on a divisor. Then $p_g(X)=0$.
\end{thm}

Since the seminal paper \cite{BS}, a plethora of variant statements and generalizations have seen the day (cf. \cite[Chapter 3]{Vo} for a recent and comprehensive overview of the field). The modest aim of this short expository note is to provide yet one more variant statement, showing that for a general variety, {\sl all} Chow groups are very large. The price to pay for starting out not with $0$--cycles but with cycles of arbitrary codimension $i$ is that we need to assume the standard Lefschetz conjecture $B(X)$. Here is the main result of this note:

\begin{thm} Let $X$ be a smooth projective variety, and suppose $B(X)$ is true. Suppose there is an $i$ such that the Chow group $A^i(X)_{\QQ}$ is supported on a divisor. Then the cohomology group $H^i(X,\QQ)$ is supported on a divisor.
\end{thm}

This can be used to provide instances of varieties for which all Chow groups are very large:

\begin{cor} Let $X$ be an abelian variety. Then no Chow group $A^i(X)_{\QQ}$ is supported on a divisor.
\end{cor}

More examples of this type are given below (Corollary \ref{cor}); the same statement holds for any variety for which one knows $B(X)$ and whose Hodge diamond is of maximal width. This is not surprising, and probably known to experts, yet we couldn't find a reference.
Closely related results appear in work of Lewis \cite{L}, \cite{L2} and Schoen \cite{S}, yet their statements (as well as the proofs) are slightly different from ours.\footnote{Both Lewis and Schoen suppose the generalized Hodge conjecture holds true, rather than ``only'' $B(X)$. Also, both work with the notion of ``representable Chow group'', rather than with the notion of ``Chow group supported on a divisor''.}

The present note was written while preparing for the Strasbourg ``groupe de travail'' based on the book \cite{Vo}. I'd like to thank the participants of this groupe de travail for a very pleasant and stimulating atmosphere.

 \begin{convention} In this note, the word {\sl variety} refers to a smooth projective algebraic variety over $\C$.
\end{convention}

\section{The Lefschetz standard conjecture}

Let $X$ be a smooth projective variety of dimension $n$, and $h\in H^2(X,\QQ)$ the class of an ample line bundle. The hard Lefschetz theorem asserts that the map
  \[  L^{n-i}\colon H^i(X,\QQ)\to H^{2n-i}(X,\QQ)\]
  obtained by cupping with $h^{n-i}$ is an isomorphism, for any $i< n$. One of the standard conjectures asserts that the inverse isomorphism is algebraic.

\begin{defn}{} For a given $i<n$, we say that $B(X,i)$ holds if for all ample $h$ the isomorphism 
  \[  (L^{n-i})^{-1}\colon 
  H^{2n-i}(X,\QQ)\stackrel{\cong}{\rightarrow} H^i(X,\QQ)\]
  is induced by a correspondence.
 \end{defn}  
 
\begin{defn}{(Lefschetz standard conjecture $B(X)$)} Following convention, we say that $B(X)$ holds if $B(X,i)$ holds for all $i=0,\ldots,n-1$.
\end{defn}

For later use, we recall the notion of geometric coniveau:

\begin{defn}{(geometric coniveau)} The geometric coniveau filtration on cohomology is defined as
  \[  N^j H^i(X,\QQ)=\sum_{Z\subset X}  \hbox{Im}\bigl( H^i_Z(X,\QQ)\to H^i(X,\QQ)\bigr)\ ,\]
  where $Z$ runs through all subschemes of $X$ of codimension $\ge j$.
  \end{defn}

We define a pool of examples for which $B(X)$ is known to hold:

\begin{defn}\label{BB} Let ${\mathcal B}$ be the class of varieties defined by the following rules:

\item{(1)} The following varieties are in $\BB$: 
\subitem{(\rom 1)} Curves and surfaces;
\subitem{(\rom2)} Threefolds not of general type (i.e. having Kodaira dimension $<3$);
\subitem{(\rom3)} Abelian varieties;
\subitem{(\rom4)} $n$--dimensional varieties $X$ which have $A_i(X)_{\QQ}$ supported on a subvariety of dimension $i+2$ for all $i\le{n-3\over 2}$;
\subitem{(\rom5)} $n$--dimensional varieties $X$ which have $H_i(X,\QQ)=N^{\llcorner {i\over 2}\lrcorner}H_i(X,\QQ)$ for all $i>n$.

\item{(2)} $\BB$ is closed under taking products, and under taking smooth hyperplane sections.
\item{(3)} $\BB$ is closed under blow--up, i.e. if $\wt{X}$ is the blow--up of $X$ with center $Y$, then $\wt{X}$ is in $\BB$ if and only if $X$ and $Y$ are in $\BB$.
\end{defn}

\begin{prop} For $X$ in $\BB$, the Lefschetz standard conjecture $B(X)$ is true.
\end{prop}

\begin{proof} For curves, surfaces and abelian varieties, this is proven by Kleiman 
\cite[\S 2 Appendix]{K0}. The case of threefolds not of general type was proven by Tankeev \cite{T}. Case (\rom4) is \cite[Theorem 7.1]{V}. Case (\rom5) follows from \cite[Theorem 4.2]{V2}.

The fact that products and hyperplane sections preserve the truth of $B(X)$ is well--known \cite{K}. The statement for blow--ups is proven in \cite{T0}.

\end{proof}

\begin{rem} Point (\rom4) of Definition \ref{BB} implies that rationally connected fourfolds are in $\BB$. It also implies that all linear varieties (as defined in \cite{Tot}) are in $\BB$; this class includes toric varieties and spherical varieties.
Point (\rom5) implies that every threefold with $h^{0,2}=0$ is in $\BB$.
\end{rem}

\begin{rem} Point (3) of Definition \ref{BB} implies the following: if $X$ is a variety of dimension $\le 4$, and $X$ is birational to a variety in $\BB$, then $X\in\BB$.
\end{rem}

\section{Main result}

\begin{thm}{\label{main}} Let $X$ be a smooth projective variety over $\C$. Suppose there is an $i$ such that the Chow group $A^i(X)_{\QQ}$ is supported on a divisor, and that $B(X,j)$ is true for $j\le i$.
 Then the cohomology group $H^i(X,\QQ)$ is supported on a divisor.
\end{thm}

This is useful in showing the following: for a general variety, all the Chow groups are as large as possible. Here a ``general variety'' means a variety having Hodge diamond of maximal width, and ``large Chow group'' means not supported on a subvariety.

\begin{cor}\label{cor} Let $X$ be a variety in $\BB$. Suppose for all $i=1,\ldots,n$ the Hodge numbers $h^{i,0}(X)$ are $\not=0$. Then there is no Chow group $A^i(X)_{\QQ}$ supported on a divisor.
\end{cor}

\begin{proof}{(of Corollary \ref{cor})} This is immediate from Theorem \ref{main}, plus the fact that $H^i(X,\C)$ being supported on a divisor implies (by functoriality of the Hodge filtration) that $H^{i,0}X$ is $0$.
\end{proof}

By way of example, we present two explicit instances of Corollary \ref{cor}:

\begin{cor} Let $X$ be an abelian variety. Then no Chow group $A^i(X)_{\QQ}$ is supported on a divisor.
\end{cor}

\begin{cor} Let $L$ be any variety in $\BB$ of dimension $m$. Let $C_1,\ldots,C_r$ be non--rational curves. Let
\[   X\subset (L\times C_1\times\cdots\times C_r)\]
be a complete intersection of codimension $\ge m$. Then no Chow group $A^i(X)_{\QQ}$ is supported on a divisor.
\end{cor}

\begin{proof}{(of Theorem \ref{main})} For $i$ equal to $n=\dim X$, this is Mumford's theorem. We will now suppose $i<n$.

 The fact that $B(X,j)$ is true for $j\le i$ implies \cite[Theorem 4-1]{K} that the K\"unneth component of the diagonal
\[  \pi_i\in \ima\Bigl(H^{2n-i}(X,\QQ)\otimes H^i(X,\QQ)\to H^{2n}(X\times X,\QQ)\Bigr)\]
is algebraic. Let $h\in H^2(X,\QQ)$ be the class of a very ample line bundle, and let
 $Y\subset X$ be a general complete intersection of dimension $i$ with class $[Y]=h^{n-i}\in H^{2n-2i}(X,\QQ)$. The weak Lefschetz theorem gives a surjection
  \[  H^i(Y,\QQ) \rightarrow H^{2n-i}(X,\QQ)\ ,\]
  implying that  actually $\pi_i$ is in the image of the composite map
  \[  \pi_i\in \ima\Bigl(H^{i}(Y,\QQ)\otimes H^i(X,\QQ)\to H^{2i}(Y\times X,\QQ)\to H^{2n}(X\times X,\QQ)\Bigr)\ .\]
  
Using Lemma \ref{csv} proved below, we can find an algebraic class
  \[ \pi_i^\prime\in H^{2i}(Y\times X,\QQ)\]
 representing $\pi_i$ (i.e. the push--forward of $\pi_i^\prime$ equals $\pi_i$ in $H^{2n}(X\times X,\QQ)$).
 We can thus lift $\pi_i^\prime$ to the Chow group $A^i(Y\times X)_{\QQ}$ and, under the assumption of Theorem \ref{main}, we can apply the Bloch--Srinivas argument \cite{BS} (in the form of Proposition \ref{BS} below) to get a decomposition
   \[ \pi_i^\prime = \Gamma_1+\Gamma_2\in A^i(Y\times X)_{\QQ},\]
   where $\Gamma_1$ (resp. $\Gamma_2$) is supported on $Y^\prime\times X$ where $Y^\prime\subset Y$ is a divisor (resp. supported on $Y\times D$, where $D\subset X$ is a divisor).
Going back to cohomology, this induces a decomposition of the K\"unneth component
  \[   \pi_i=\Gamma_1+\Gamma_2\in  H^{2n}(X\times X,\QQ),\]
with $\Gamma_1,\Gamma_2$ as above.
We now consider the action of the correspondence $\pi_i$ on 
%the graded for the coniveau filtration $N$ on 
$H^i(X,\QQ)$: 
%(that this is possible is the content of \cite{AK}):
  \[  \hbox{id}=(\pi_i)_\ast=(\Gamma_1)_\ast+(\Gamma_2)_\ast\colon  H^i(X,\QQ)\to  H^i(X,\QQ)\ .\]

Because $\Gamma_2$ is supported on $Y\times D$, we obviously have
  \[  (\Gamma_2)_\ast H^i(X,\QQ)\subset H^i_D(X,\QQ)\subset N^1 H^i(X,\QQ)\ .\]    
As for $\Gamma_1$, we have that $\Gamma_1$ is supported on $Y^\prime\times X$. This means that the action
of $\Gamma_1$ factors over $H^i(\wt{Y^\prime},\QQ)$, where $\wt{Y^\prime}\to Y^\prime$ is a resolution of singularities. Now, 
$H^i(\wt{Y^\prime},\QQ)$ is supported on a divisor $Y^{\prime\prime}$ (weak Lefschetz for the $(i-1)$--dimensional variety $\wt{Y^\prime}$).
It follows that the image of $H^i(\wt{Y^\prime},\QQ)$ in $H^i(X,\QQ)$ via the correspondence $\Gamma_1$ is supported on the divisor
\[ pr_2\bigl( \hbox{Supp}(\Gamma_1)\cap (pr_1)^{-1}(Y^{\prime\prime})\bigr)\subset X\ ,\] 
and hence
  \[  (\Gamma_1)_\ast H^i(X,\QQ)\subset N^1 H^i(X,\QQ)\ .\]
  \end{proof}

\begin{lem}\label{csv} Let $X$ be a variety for which $B(X,i)$ holds. Let $Y\subset X$ a smooth complete intersection of dimension $i$ and with 
$[Y]=h^{n-i}\in H^{2n-2i}(X,\QQ)$, for some $h\in H^2(X,\QQ)$ the class of an ample line bundle. Suppose a class
  \[ c\in   \ima\Bigl( H^i(Y,\QQ)\otimes H^i(X,\QQ)\to H^{2i}(Y\times X,\QQ)\to H^{2n}(X\times X,\QQ)\Bigr)  \]
is algebraic. Then there exists an algebraic class 
  \[ c^\prime \in H^{2i}(Y\times X,\QQ)\]
  representing $c$.
  \end{lem}
  
\begin{proof} Let $\gamma$ denote the correspondence inducing the isomorphism
  \[  H^{2n-i}(X,\QQ)   \stackrel{\cong}{\rightarrow} H^i(X,\QQ)\ \]
  inverse to the cup--product with $h^{n-i}$,
  and let $\Delta\in H^{2n}(X\times X,\QQ)$ denote the class of the diagonal.
Then $\gamma\times\Delta$ is a correspondence that acts
  \[  (\gamma\times\Delta)_\ast\colon \ima\Bigl( H^{2n-i}X\otimes H^iX\to H^{2n}(X\times X)\Bigr) \rightarrow
     \ima \Bigl(    H^{i}X\otimes H^iX\to H^{2i}(X\times X)\Bigr) \ .\]
  It follows that
    \[c^{\prime\prime}:= (\gamma\times\Delta)_\ast c\in H^{2i}(X\times X)\]
    is algebraic. But then, denoting $j\colon Y\to X$ the inclusion, the restriction
    \[  c^\prime:= (j\times\hbox{id})^\ast(c^{\prime\prime})\in H^{2i}(Y\times X)\]
    is algebraic as well. But the composition $j_\ast j^\ast$ is cup--product with $h^{n-i}$ on cohomology, so that
      \[  j_\ast j^\ast\gamma_\ast=\hbox{id}\colon\   H^{2n-i}(X,\QQ)\ \to\ H^{2n-i}(X,\QQ)\ .\]
   This implies that
     \[\begin{split}   (j\times\hbox{id})_\ast (c^\prime)&=(j\times\hbox{id})_\ast (j\times\hbox{id})^\ast (\gamma\times\hbox{id})_\ast c\\
                                                                    &=(j_\ast j^\ast \gamma_\ast\times\hbox{id}) c\\
                                                                    &=c\ \ \ \in H^{2n}(X\times X,\QQ)\ .
                                                                    \end{split}\]
  \end{proof}

\begin{rem} Lemma \ref{csv} is a particular case of Voisin's standard conjecture \cite[Conjecture 2.29]{Vo}, \cite{V0}. That it can be proven here is because we suppose $B(X,i)$; this is a particular instance of the fact that the Lefschetz standard conjecture implies Voisin's standard conjecture \cite[Proposition 2.32]{Vo}.
\end{rem}

\begin{prop}{(Bloch--Srinivas-style)}\label{BS} Let $X$ be a smooth projective variety of dimension $n$. Suppose that for some $i=1,\ldots,n$, the Chow group
$A^i(X)_{\QQ}$ is supported on a subvariety $Z\subset X$. Then for any variety $Y$, and any cycle $\pi\in A^i(X\times Y)_{\QQ}$, there is a decomposition
  \[  \pi=\Gamma_1+\Gamma_2 \in A^i(X\times Y)_{\QQ}\ ,\]
where $\Gamma_1$ is supported on $Z\times Y$, and $\Gamma_2$ is supported on $X\times Y^\prime$, where $Y^\prime$ is a divisor on $Y$.
\end{prop}

\begin{proof} We may suppose everything ($X$, $Z$, $Y$ and $\pi$) is defined over a field $k\subset\C$ which is finitely generated over its prime subfield. Let $k(Y)$ denote the function field of $Y$. Since $k(Y)\subset\C$, we have a map
  \[  A^i(X_{k(Y)})_{\QQ}\to A^i(X_{\C})_{\QQ}\ ,\]
  which is injective \cite[Appendix to Lecture 1]{B}. Hence 
the hypothesis on $A^i(X_{\C})_{\QQ}$ implies that $A^i(X_{k(Y)})_{\QQ}$ is supported on the subvariety $Z$. On the other hand,
  \[ A^i(X_{k(Y)})_{\QQ}=   \varinjlim A^i(X\times U)_{\QQ}\ ,\]
 where the limit is taken over opens $U\subset Y$ \cite[Appendix to lecture 1]{B}. Given the cycle $\pi$, we can thus find an open $j\colon U_0\subset Y$ such that the restriction $j^\ast\pi$ equals some cycle $\gamma$ on $Z\times U_0$:
   \[  j^\ast\pi=\gamma\in A^i(X\times U_0)_{\QQ}\ .\]   
Now defining $\Gamma_1\in A^i(X\times Y)_{\QQ}$ to be any cycle supported on $Z\times Y$ that restricts to $\gamma$, this means we have
  \[  j^\ast(\pi-\Gamma_1) =0\in   A^i(X\times U_0)_{\QQ}\ ;\]
  i.e. the difference $\Gamma_2:=\pi-\Gamma_1$ is supported on $X\times (Y\setminus U_0)$.  
\end{proof}  
  
  \begin{rem} The Bloch--Srinivas style Proposition \ref{BS} can also be deduced as a special case of Voisin's presentation \cite[Theorem 3.1]{Vo} of the decomposition principle.
  \end{rem}

\begin{rem} It is established in \cite{AK} that correspondences act on gradeds of the geometric coniveau filtration. This gives another way of concluding, at the end of the proof of Theorem \ref{main}, that 
  \[ (\Gamma_1)_\ast H^i(X,\QQ)\subset N^1 H^i(X,\QQ)\ .\]
The more direct and explicit argument presented above was suggested by the anonymous referee, to whom we are grateful.
\end{rem}

% ------------------------------------------------------------------------

% ------------------------------------------------------------------------

\begin{thebibliography}{1}

\bibitem{A} D. Arapura, Motivation for Hodge cycles, Advances in Math. (2006),

\bibitem{AK} D. Arapura and S.-J. Kang, Functoriality of the coniveau filtration, Canad. Math. Bull. (2007),

\bibitem{B} S. Bloch, Lectures on algebraic cycles, Duke Univ. Math. Series, Vol. IV, 

\bibitem{BS} S. Bloch and V. Srinivas, Remarks on correspondences and algebraic cycles, American Journal of Mathematics Vol. 105, No 5 (1983), 1235---1253,

\bibitem{K0} S. Kleiman, Algebraic cycles and the Weil conjectures, in: Dix expos\'es sur la cohomologie des sch\'emas, North--Holland Amsterdam, 1968, 359---386, 

\bibitem{K} S. Kleiman, The standard conjectures, in: Motives (U. Jannsen et alii, eds.), Proceedings of Symposia in Pure Mathematics Vol. 55 (1994), Part 1, 

\bibitem{L} J. Lewis, Towards a generalization of Mumford's theorem, J. Math. Kyoto Univ. 29 (1989), 267---272,

\bibitem{L2} J. Lewis, A generalization of Mumford's theorem, II, Illinois Journal of Mathematics Vol. 39 No 2 (1995), 288---304,

\bibitem{M} D. Mumford, Rational equivalence of $0$--cycles on surfaces, J. Math. Kyoto Univ. Vol. 9 No 2 (1969), 195---204,

\bibitem{S} C. Schoen, On Hodge structures and non--representability of Chow groups, Comp. Math. 88 (1993), 285---316,

\bibitem{T0} S. Tankeev, Monoidal transformations and conjectures on algebraic cycles, Izvestiya Math. 71 (2007), no. 3, 629---655,

\bibitem{T} S. Tankeev, On the standard conjecture of Lefschetz type for complex projective threefolds. II, Izvestiya Math. 75:5 (2011), 1047---1062,

\bibitem{Tot} B. Totaro, Chow groups, Chow cohomology, and linear varieties, Forum of Mathematics, Sigma (2014), vol. 1, e1,

\bibitem{V} C. Vial, Algebraic cycles and fibrations, Documenta Math. 18 (2013), 1521---1553,

\bibitem{V2} C. Vial, Projectors on the intermediate algebraic Jacobians, New York J. Math. 19 (2013), 793---822,

\bibitem{V0} C. Voisin, The generalized Hodge and Bloch conjectures are equivalent for general complete intersections, 
Annales scientifiques de l'ENS 46, fascicule 3 (2013), 449-475,

\bibitem{Vo} C. Voisin, Chow Rings, Decomposition of the Diagonal, and the Topology of Families, Princeton University Press, Princeton and Oxford, 

\end{thebibliography}
\end{document}